\documentclass[10pt, reqno]{amsart}
\usepackage{graphicx, amssymb, amsmath, amsthm}
\usepackage{epsfig}
\usepackage{hyperref}
\numberwithin{equation}{section}

\def\ov{\overline}

\let\Re=\undefined\DeclareMathOperator*{\Re}{Re}
\let\Im=\undefined\DeclareMathOperator*{\Im}{Im}

\newcommand{\R}{\mathbb{R}}
\newcommand{\C}{\mathbb{C}}

\newtheorem{theorem}{Theorem}[section]

\newtheorem{lemma}[theorem]{Lemma}

\newtheorem{corollary}[theorem]{Corollary}

\newtheorem{proposition}[theorem]{Proposition}

\theoremstyle{definition}

\newtheorem{remark}[theorem]{Remark}

\makeatletter
\newcommand{\Extend}[5]{\ext@arrow0099{\arrowfill@#1#2#3}{#4}{#5}}
\makeatother

\begin{document}
\title[Focusing NLS ]{Scattering Below Ground State of Focusing Fractional Nonlinear Sch\"{o}rdinger Equation with Radial Data }

\author[C. Sun]{Chenmin Sun}
\address{Universit\'e C\^{o}te d'Azur, LJAD, France}
\email{csun@unice.fr}
\author[H.Wang]{Hua Wang}
\address{Department of Mathematics, Central China Normal University, Wuhan, Hubei 430079, China}
\email{wanghua math@126.com;}
\author[X.Yao]{Xiaohua Yao}
\address{Hubei Key Laboratory of Mathematical Sciences and School of Mathematics and Statistics, Central China Normal University, Wuhan, 430079, P.R. China}
\email{yaoxiaohua@mail.ccnu.edu.cn}
\author[J. Zheng]{Jiqiang Zheng}
\address{Universit\'e C\^{o}te d'Azur, LJAD, France}
\email{zhengjiqiang@gmail.com}

 \maketitle

\begin{abstract}
	The aim of this paper is to adapt the strategy in \cite{Dodson} [ See, B. Dodson, J. Murphy, a new proof of scattering below the ground state for the 3D radial focusing cubic NLS, arXiv:1611.04195 ] to prove the scattering of radial solutions below sharp threshold for certain focusing fractional NLS. The main ingredient is to apply the fractional virial identity proved in \cite{Lenzmann} [ See, T. Boulenger, D. Himmelsbach, E. Lenzmann, Blow up for fractional NLS,J. Func. Anal, 271(2016), 2569-2603 ] to exclude the concentration of mass near the origin.
\end{abstract}

\begin{center}
 \begin{minipage}{100mm}
   { \small {{\bf Key Words:}  fractional Schr\"odinger equation;  scattering; Morawetz estimate.}
      {}
   }\\
    { \small {\bf AMS Classification:}
      {35P25,  35Q55, 47J35.}
      }
 \end{minipage}
 \end{center}

\section{Introduction}

\noindent

\noindent In this paper we study the initial-value problem for focusing fractional
nonlinear Schr\"odinger equations(FNLS)
\begin{align} \label{equ:equ1.1}
\begin{cases}    i\partial_tu-(-\Delta)^su= -|u|^{p-1} u,\quad
(t,x)\in\R\times\R^d,
\\
u(0,x)=u_0(x),
\end{cases}
\end{align}
where $u:\R_t\times\R_x^d\to \C$ and the index $s\in(0,1)$.

The fractional Sch\"{o}rdinger equation is a fundamental model coming from fractional quantum mechanics, which was derived by Laskin as a result of extending the Feynman path integral, from the Brownian-like to L\'{e}vy-like quantum mechanical maths. See for example \cite{Laskin}.
Suppose $1+\frac{4s}{d}<p<1+\frac{4s}{d-2}$, and then the critical exponent
$$ s_c=\frac{d}{2}-\frac{2s}{p-1}\in (0,s).
$$
The critical exponent comes from the scaling
$$ u(t,x)\mapsto \lambda^{\frac{2s}{p-1}}u(\lambda^{2s}t,\lambda x)
$$
which keeps \eqref{equ:equ1.1} invariant. Moreover,
$$ \|\lambda^{\frac{2s}{p-1}}u(0,\lambda\cdot)\|_{\dot{H}^{s_c}(\mathbb{R}^d)}=\|u(0,\cdot)\|_{\dot{H}^{s_c}(\mathbb{R}^d)}.$$

There are two basic conserved quantities for the flow \eqref{equ:equ1.1}:
$$ \textrm{Mass }: M[u]=\int_{\mathbb{R}^d}|u(t,x)|^2dx,
$$
$$ \textrm{Energy }:
E[u]=\frac{1}{2}\int_{\mathbb{R}^d}|(-\Delta)^{s/2}u(t,x)|^2dx-\frac{1}{p+1}\int_{\mathbb{R}^d}|u(t,x)|^{p+1}dx.
$$

\eqref{equ:equ1.1} has been widely studied in the recent years. For the aspect of local Cauchy theory and small data scattering, see for example \cite{Cho} \cite{Hong}. \eqref{equ:equ1.1} has good dispersion property when $s>\frac{1}{2}$. For $s=1$, \eqref{equ:equ1.1} is just nonlinear Schr\"{o}dinger equation, while for $s=\frac{1}{2}$ is called the nonlinear half-wave equation. See \cite{Raphael} for the travelling waves and blow-up dynamics of $1D$ cubic half-wave equation.

The blow-up or long-time dynamics of the system \eqref{equ:equ1.1} turns out to be a very interesting problem. In \cite{Lenzmann}, the authors have proved the existence of blow-up dynamics for radial solutions, subject to certain threshold (see Theorem \ref{scattering} below). For energy critical FNLS $(s_c=s)$, the authors in \cite{Guo2} have performed Kenig-Merle type analysis (see \cite{KM}) to prove the global well-posedness of radial solutions and scattering below sharp threshold of stationary solutions.

In this paper, we will study the long-time dynamics for radial solutions of \eqref{equ:equ1.1} below the sharp threshold (see \cite{Lenzmann}).

In the energy sub-critical case ($s_c<s$), there exist solitary waves solutions of the form $u_Q(t,x)=Q(x)e^{it}$, where $Q$ is radial $H^s$ function which solves the fractional elliptic equation:
$$ (-\Delta)^sQ+Q-Q^{p}=0.
$$
We remark that the existence of $Q$ can be derived from variational analysis just as the case $s=1$, while the uniqueness issue is much more difficult to prove, see \cite{Lenzmann2}.

Note that $Q$ is an extremum of the sharp fractional Gagliardo-Nirenberg inequality:
$$ \|u\|_{L^{p+1}(\mathbb{R}^d)}^{p+1}\leq C(d,p,s)\|(-\Delta)^{s/2}u\|_{L^2(\mathbb{R}^d)}^\frac{d(p-1)}{2s}\|u\|_{L^2(\mathbb{R}^d)}^{p+1-\frac{d(p-1)}{2s}}.
$$
The main result of this note is the following:

\begin{theorem}\label{scattering}
Assume that $d\geq 3$, $s\in\left(\frac{d}{d+1},1\right)$, $\frac{8s}{3}<p<1+\frac{4s}{3-2s}$ when $d=3$ and $2\leq p<1+\frac{4s}{d-2s}$ when $d\geq 4$. Assume $u_0\in H^s(\mathbb{R}^d)$ is radial and $E[u_0]\geq 0$. Moreover, suppose that
$$
E[u_{0}]^{s_{c}}M[u_{0}]^{s-s_{c}}<E[Q]^{s_{c}}M[Q]^{s-s_{c}},
$$

$$
\|(-\Delta)^{\frac{s}{2}}u_{0}\|_{L^{2}}^{s_{c}}\|u_{0}\|_{L^{2}}^{s-s_{c}}
<\|(-\Delta)^{\frac{s}{2}}Q\|_{L^{2}}^{s_{c}}\|Q\|_{L^{2}}^{s-s_{c}}.
$$
Then the solution $u(t)$ to \eqref{equ:equ1.1} is globally well-posed and scatters in the sense:
$$ \lim_{t\rightarrow \pm\infty}\|u(t)-e^{-it(-\Delta)^s}u_{\pm}\|_{H^s(\mathbb{R}^d)}=0
$$
for some $u_+,u_-\in H^s(\mathbb{R}^d)$.
\end{theorem}
\begin{remark}
	One can check that under the assumption on $p$, the condition $0<s<s_c$ holds automatically. The additional assumption on $p$ is technical since we need prevent $p$ to be too small in the proof of Lemma \ref{scacri}. Hence there is no room of validity of our theorem when the dimension $d$ is greater than $6$.
\end{remark}

\begin{remark}
	In \cite{Lenzmann}, the authors have proved the existence of blow up dynamics of \eqref{equ:equ1.1} under the condition either $E[u_0]<0$ or $E[u_0]\geq 0$ and
	$$ E[u_0]^{s_c}M[u_0]^{s-s_c}<E[Q]^{s_c}M[Q]^{s-s_c},
	$$
	$$ \|(-\Delta)^{s/2}u_0\|_{L^2(\mathbb{R}^3)}^{s_c}\|u_0\|_{L^2(\mathbb{R}^3)}^{s-s_c}>
	\|(-\Delta)^{s/2}Q\|_{L^2(\mathbb{R}^d)}^{s_c}\|Q\|_{L^2(\mathbb{R}^d)}^{s-s_c}.
	$$
	Our result is a complement of this blow up result. This coincides with the viewpoint of NLS ($s=1$). More precisely, when $s_c<0$, we expect the orbital stability of solitary waves, and $0\leq s_c\leq s$, we always expect the scattering below this sharp threshold, compared to the NLS ($s=1$).
\end{remark}

\begin{remark}
The reason why we only deal with radial symmetric solutions is technical. For $s<1$, the Strichartz estimate for $e^{-it(\Delta)^s}$ will happen to loss of derivative. However, we will have full range of Strichartz admissible pairs when restricting to radial symmetric functions, see \cite{Guo1}.
\end{remark}

\begin{remark}
There are several natural questions. The first one concerns about dropping the restriction of radial symmetry of the initial data.  Furthermore, we wonder the characterization of solutions of \eqref{equ:equ1.1} in the mass-critical case, namely $p=1+\frac{4s}{d}$. Last but not least, the classification of solutions at the ground state level, namely under the condition $ E[u_0]^{s_c}M[u_0]^{s-s_c}=E[Q]^{s_c}M[Q]^{s-s_c}.$ These problems will be considered in the forthcoming work.
\end{remark}

This paper is organized as following: We mainly follow the strategy of \cite{Dodson}. After introducing basic notations and preliminaries, we prove a scattering criterion for FNLS of radial solutions in section 3, which is the generalization
of the NLS. In section 4 we prove Morawetz estimate with the aid of fractional version of virial identity. We also add several appendices, including the formal derivation of virial identity, in order to make this note more self-contained. Furthermore, we also discuss the defocusing FNLS briefly in the appendix.


\section{Notations and Preliminaries }

We use the notation
$$ \|f\|_{L^r(\mathbb{R}^d)}=\left(\int_{\mathbb{R}^d}|f(x)|^rdx\right)^{\frac{1}{r}},\;
\|f\|_{L_t^qL_x^r(I\times\mathbb{R}^d)}=\left(\int_{I}\left(\int_{\mathbb{R}^d}|f(x)|^rdx\right)^{\frac{q}{r}}dt\right)^{\frac{1}{q}}.
$$
For simplicity, $\|f\|_{L_t^qL_x^r(I\times\mathbb{R}^d)}$ can be written as $\|f\|_{L_I^qL_x^r}$.
And Sobolev norms can be defined as
$$ \|f\|_{\dot{H}_x^{\sigma}}:=\|D^{\sigma}f\|_{L_x^2},\; D^{\sigma}:=(-\Delta)^{\sigma/2}.
$$

For the Littlewood-Paley decomposition, we shall adapt the notations in \cite{Chemin}: Let
$$ u=\sum_{j\in\mathbb{Z}}\dot{\Delta}_j u,$$
be the homogeneous Littlewood-Paley decomposition for $u\in\mathcal{S}'_h(\mathbb{R}^d)$.
In addition, denote
$$ \dot{\Delta}_{\leq 0}=\sum_{j\leq 0}\dot{\Delta}_j,\dot{\Delta}_{>0}=\sum_{j>0}\dot{\Delta}_j
$$

Now we recall the following classical fractional radial Strauss inequality, which plays a crucial role in this note.
\begin{lemma}[\cite{FStrauss}]\label{Strauss}
For all radial functions $u\in \dot{H}^s(\mathbb{R}^d)$, we have
$$\||x|^{\frac{d-2s}{2}}u\|_{L^{\infty}(\mathbb{R}^d)}\lesssim \|u\|_{\dot{H}^s},\forall u\in \dot{H}_{rad}^s(\mathbb{R}^d).
$$
\end{lemma}

We also need the following fractional Leibniz rule, proved in \cite{KPV}:
\begin{lemma}\label{KPV}
	Suppose $0<s<1,s_1,s_2\in [0,s],s=s_1+s_2$, and $p,p_1,p_2\in (1,\infty)$ satisfies
	$$ \frac{1}{p}=\frac{1}{p_1}+\frac{1}{p_2},
	$$
	then
	$$ \|D^s(fg)-fD^sg-gD^sf\|_{L_x^p}\lesssim \|D^{s_1}f\|_{L_x^{p_1}}\|D^{s_2}g\|_{L_x^{p_2}},
	$$
	and the inequality still holds true for $s_1=0,p_1=\infty$.
\end{lemma}

An important tool is the Strichartz estimates for radial solutions, obtained in \cite{Guo1}:

We first introduce several notations: We say a pair $(q,r)\in \Lambda_d(0)$ is $s-$admissible, if
$$ \frac{2s}{q}=d\left(\frac{1}{2}-\frac{1}{r}\right), (q,r)\in [2,\infty]^2.
$$
And the admissible pair of level $\gamma\geq 0$ can be defined by the following relation:
$$ (q,r)\in \Lambda_d(\gamma)\Longleftrightarrow
\frac{2s}{q}=d\left(\frac{1}{2}-\frac{1}{r}\right)-\gamma, (q,r)\in [2,\infty]^2.
$$
For brevity, define $A_s=(-\Delta)^s$ and $e^{-itA_s}=e^{-it(-\Delta)^s}$ in the sequel.
\begin{lemma}[\cite{Guo1}]\label{Strichartz}
	Suppose $d\geq 2$ and $s\in\left(\frac{d}{2d-1},1\right)$.

	Then we have for any $(q,r),(\tilde{q},\tilde{r})\in\Lambda_d(0)$,
\begin{equation}
\label{Strichartz1}
 \|e^{-itA_s}u_0\|_{L_t^qL_x^r(\mathbb{R}\times\mathbb{R}^d)}\lesssim\|u_0\|_{L^2(\mathbb{R}^d)},
\end{equation}
and
\begin{equation}
\label{Strichartz2}
 \left\|\int_{0}^te^{-i(t-t')A_s}F(t')dt'\right\|_{L_t^qL_x^r(\mathbb{R}\times\mathbb{R}^d)}
\lesssim \|F\|_{L_t^{\tilde{q}'}L_x^{\tilde{r}'}(\mathbb{R}\times\mathbb{R}^d)},
\end{equation}
\begin{equation}\label{Strichartz3}
\left\|\int_{\mathbb{R}}e^{-i(t-t')A_s}F(t')dt'\right\|_{L_t^qL_x^r(\mathbb{R}\times\mathbb{R}^d)}
\lesssim \|F\|_{L_t^{\tilde{q}'}L_x^{\tilde{r}'}(\mathbb{R}\times\mathbb{R}^d)},
\end{equation}
\begin{equation}\label{Strichartz4}
\left\|\int_{\mathbb{R}}e^{-itA_s}g(t,\cdot)dt\right\|_{L_x^2}
\lesssim \|g\|_{L_t^{\tilde{q}'}L_x^{\tilde{r}'}(\mathbb{R}\times\mathbb{R}^d)},
\end{equation}
where for real number $a\in[1,\infty]$, $a'$ satisfies
$$ \frac{1}{a}+\frac{1}{a'}=1.
$$
Moreover, for $(q_1,r_1)\in \Lambda_d(s)$, we have
\begin{equation}\label{Strichartz5}
 \left\|\int_0^t e^{-i(t-t')A_s}F(t')dt'\right\|_{L_t^{q_1}L_x^{r_1}}\lesssim \|F\|_{L_t^{q_2'}L_x^{r_2'}}
\end{equation}
with
$$ \frac{2s}{q_2}=d\left(\frac{1}{2}-\frac{1}{r_2}\right)+s.
$$
\end{lemma}

For the proof of Lemma \ref{scacri} in the next section, we will also need a dispersive estimate for the semi-group $e^{-itA_s}$:
\begin{lemma}\label{dispersive}
For $\frac{1}{2}<\alpha<1,r\in [2,\infty]$, we have
$$ \left\|\dot{\Delta}_ke^{-itA_{\alpha}}f\right\|_{L^r(\mathbb{R}^d)}\lesssim 2^{kd(1-\alpha)\left(1-\frac{2}{r}\right)}|t|^{-\frac{d}{2}\left(1-\frac{2}{r}\right)}\|\dot{\Delta}_kf\|_{L^{r'}(\mathbb{R}^d)},
$$
for all $k\in\mathbb{Z}$.
\end{lemma}

 From standard stationary phase analysis, we have the following dispersive estimate:

\begin{lemma}[\cite{CMY}]\label{dispersiveestimate}
	
	If $0<s\neq\frac{1}{2}, d\geq 2$, Then
	
	(i) If $\beta>-d$, then
	$$
	\|(-\Delta)^{\frac{\beta}{2}}e^{-it(-\Delta)^{s}}\dot{\Delta}_{\leq 0}u_{0}\|_{L^{\infty}}
	\lesssim |t|^{-\theta}\|u_{0}\|_{L^{1}},
	0\leq\theta\leq \min\{\frac{d+\beta}{2s},\frac{d}{2}\}.
	$$

	(ii) If $\beta>-d$ and $\frac{d+\beta}{2s}\leq\frac{d}{2}$,
	$$
	\|(-\Delta)^{\frac{\beta}{2}}e^{-it(-\Delta)^{s}}\dot{\Delta}_{> 0}u_{0}\|_{L^{\infty}}
	\lesssim |t|^{-\theta}\|u_{0}\|_{L^{1}},
	\frac{d+\beta}{2s}\leq\theta\leq \frac{d}{2}.
	$$
\end{lemma}

In particular, taking $\beta=d(s-1)$ in the above lemma gives
$$
\|e^{-it(-\Delta)^{s}}u_{0}\|_{L^{\infty}}
\lesssim |t|^{-\frac{d}{2}}\|(-\Delta)^{\frac{d(1-s)}{2}}u_{0}\|_{L^{1}},
$$
which interpolates with
$$
\|e^{-it(-\Delta)^{s}}u_{0}\|_{L^{2}}
=\|u_{0}\|_{L^{2}}
$$
leads to the dispersive estimate
\begin{equation}\label{usualdispersive}
\|e^{-it(-\Delta)^{s}}u_{0}\|_{L^{p}}
\lesssim |t|^{-\frac{d}{2}(1-\frac{2}{p})}
\|(-\Delta)^{\frac{d(1-s)}{2}(1-\frac{2}{p})}u_{0}\|_{L^{p'}}
\end{equation}
for any $2\leq p\leq \infty$. We note that there is smoothness loss in the
above dispersive inequalities.

It follows from Lemma 2.2 that
$$
\|e^{-it(-\Delta)^{s}}f\|_{L_{t}^{2}L_{x}^{\frac{2d}{d-2s}}}
\lesssim \|f\|_{L^{2}},
$$
and hence
$$
\|e^{-it(-\Delta)^{s}}f\|_{L_{t}^{2}W_{x}^{s, \frac{2d}{d-2s}}}
\lesssim \|f\|_{H^{s}}.
$$
For $\sigma<\frac{2s}{d-2s}$, the Sobolev embedding yields
\begin{equation}\label{Linfiniteinspace1}
\|e^{-it(-\Delta)^{s}}f\|_{L_{t}^{2\sigma}L_{x}^{\infty}}
\lesssim \|f\|_{H^{s}}.
\end{equation}

The following variant of these estimates is also needed:

\begin{lemma}
We have
\begin{equation}\label{variantdispersive}
\|e^{-it(-\Delta)^{s}}f\|_{L_{t}^{\frac{4s}{d-2s}}L_{x}^{\infty}}
\lesssim \|f\|_{\dot{H}^{s}}.
\end{equation}
\end{lemma}

\begin{proof}
\eqref{Linfiniteinspace1} and Bernstein's inequality imply that for any $j$
	\begin{equation}
	\label{tobeinterpolated1}
	\|e^{-it(-\Delta)^{s}}\dot{\Delta}_{j}f\|_{L_{t}^{2\sigma}L_{x}^{\infty}}
	\lesssim 2^{j\frac{d\sigma-2s}{2\sigma}}\|\dot{\Delta}_{j}f\|_{L^{2}}
	\sim 2^{j(\frac{d\sigma-2s}{2\sigma}-s)}\|\dot{\Delta}_{j}f\|_{\dot{H}^{s}}.
	\end{equation}
	While using \eqref{tobeinterpolated1} and Bernstein's inequality, we have
	\begin{equation}\label{tobeinterpolated2}
	\|e^{-it(-\Delta)^{s}}\dot{\Delta}_jf\|_{L_{t}^{\infty}L_{x}^{\infty}}
	\lesssim 2^{j\frac{d}{2}}\|\dot{\Delta}_{j}f\|_{L^{2}}
	\sim 2^{j\frac{d-2s}{2}}\|\dot{\Delta}_jf\|_{\dot{H}^{s}}.
    \end{equation}
	Applying Marcinkeiwicz interpolation to \eqref{tobeinterpolated1} and \eqref{tobeinterpolated2} gives the Lorentz space
	estimates
	$$
	\|e^{-it(-\Delta)^{s}}\dot{\Delta}_{j}f\|_{L_{t}^{\frac{4s}{d-2s},1}L_{x}^{\infty}}
	\lesssim \|\dot{\Delta}_{j}f\|_{\dot{H}^{s}},
	$$
	and then using the triangle inequality gets
	\begin{equation}\label{tobeinterpolated3}
	\|e^{-it(-\Delta)^{s}}f\|_{L_{t}^{\frac{4s}{d-2s},1}L_{x}^{\infty}}
	\lesssim \sum_{j\in{\bf{Z}}}\|P_{j}f\|_{\dot{H}^{s}}.
    \end{equation}

	On the other hand it follows from \eqref{tobeinterpolated1} and \eqref{tobeinterpolated2} that
	$$
	\displaystyle\int_{E}\|e^{-it(-\Delta)^{s}}\dot{\Delta}_{j}f\|_{L_{x}^{\infty}}
	\lesssim \min(|E|^{1-\frac{1}{2\sigma}}2^{j(\frac{d\sigma-2s}{2\sigma}-s)},
	|E|2^{\frac{d-2s}{2}j})\|\dot{\Delta}_{j}f\|_{\dot{H}^{s}}
	$$
	for any measurable set $E\subset\mathbb{R}$. Summing this in $j$ gives
	$$
	\displaystyle\int_{E}\|e^{-it(-\Delta)^{s}}f\|_{L_{x}^{\infty}}
	\lesssim |E|^{\frac{6s-d}{4s}}\sup_{j}\|\dot{\Delta}_{j}f\|_{\dot{H}^{s}}
	$$
	and hence we have the weak-type estimate
	$$
	\|e^{-it(-\Delta)^{s}}f\|_{L_{t}^{\frac{4s}{d-2s},\infty}L_{x}^{\infty}}
	\lesssim \sup_{j}\|\dot{\Delta}_{j}f\|_{\dot{H}^{s}}.
	$$
	Interpolating this inequality with \eqref{tobeinterpolated3} gives \eqref{variantdispersive}.
\end{proof}


\section{Scattering Criterion for fractional NLS}

In this section, we generalize the scattering criterion in \cite{Dodson}, originally established in \cite{Tao1} to the fractional NLS. Roughly speaking, for radial solutions with mass supercritical and energy subcritical nonlinearity, the only obstacle to the scattering (more precisely, asymptotic completeness) is the concentration of mass in a ball centered by origin after long time evolution.

Introduce the notation
$$ \|u\|_{X^{\alpha}(I)}:=\sup_{(q,r)\in \Lambda_d(0)}\|D^{\alpha}u\|_{L_I^qL_x^r}.
$$
Note that we say a global solution $u$ to \eqref{equ:equ1.1} by means of the mild solution
$$ u(t)=e^{-it(-\Delta)^s}u_0+i\int_0^te^{-i(t-t')(-\Delta)^s}(|u|^{p-1}u)(t')dt',
$$
with $u\in C(\mathbb{R},H^s(\mathbb{R}^d))\cap X^{s}(I)$ for any finite interval $I\subset\mathbb{R}$.

In the following lemma, we shall extend the above Strichartz estimates
to nonlinear local-in-time forms.

\begin{lemma}\label{localintimelemma}
Let $u$ be a solution to (1.1) satisfying the uniform energy bound
	$$
	\|u\|_{L_{t}^{\infty}H_{x}^{s}}\leq E.
	$$
	Then we have
	\begin{equation}\label{localtime1}
	\|u\|_{L_{t}^{2}W_{x}^{s, \frac{2d}{d-2s}}
		([T, T+\tau]\times {{\bf{R}}^{d}})}\lesssim \langle\tau\rangle^{\frac{1}{2}},
	\end{equation}
    \begin{equation}
    \label{localtime2}
    \|u\|_{L_{t}^{p-1}L_{x}^{\infty}
    	([T, T+\tau]\times {{\bf{R}}^{d}})}\lesssim \langle\tau\rangle^{\frac{1}{p-1}},
    \end{equation}
    \begin{equation}
    \label{localtime3}
    \|u\|_{L_{t}^{\frac{4s}{d-2s}}L_{x}^{\infty}
    	([T, T+\tau]\times {{\bf{R}}^{d}})}\lesssim \langle\tau\rangle^{\frac{d-2s}{4s}},
    \end{equation}
	for all $T\geq 0$ and $\tau>0$.
\end{lemma}
\begin{proof}It is sufficient to consider $\tau$ sufficiently small
depending on $E$ because the result for larger times span $\tau$ then
follows by decomposing the time interval into smaller pieces. To prove
\eqref{localtime1}, \eqref{localtime2} and \eqref{localtime3}, we define
the local in time norm $$
\|u\|_{Y}:= \|u\|_{L_{t}^{2}W_{x}^{s, \frac{2d}{d-2s}}
	([T, T+\tau]\times {{\bf{R}}^{d}})}+
\|u\|_{L_{t}^{p-1}L_{x}^{\infty}
	([T, T+\tau]\times {{\bf{R}}^{d}})}+
\|u\|_{L_{t}^{\frac{4s}{d-2s}}L_{x}^{\infty}
	([T, T+\tau]\times {{\bf{R}}^{d}})}.
$$

Using Strichartz estimates Lemma \ref{Strichartz}, \eqref{Linfiniteinspace1}, \eqref{variantdispersive} and the Duhamel formula
$$
u(t)=e^{-i(t-T)(-\Delta)^{s}}u(T)
-i\displaystyle\int_{T}^{t}e^{-i(t-t')(-\Delta)^{s}}(|u(t')|^{p-1}u(t'))dt'
$$
yield
$$
\|u\|_{Y}\lesssim \|u(T)\|_{H^{s}}+\displaystyle\int_{T}^{T+\tau}\||u|^{p-1}u\|_{H^{s}}dt'.
$$
Now we apply Leibnitz rule and H\"{o}lder inequality to get
$$
\|u\|_{Y}\lesssim 1+\tau^{1-\frac{(d-2s)(p-1)}{4s}}\|u\|_{L_{t}^{\frac{4s}{d-2s}}L_{x}^{\infty}
	([T, T+\tau]\times {{\bf{R}}^{d}})}^{p-1}
\|u\|_{L_{t}^{\infty}H_{x}^{s}([T, T+\tau]\times {{\bf{R}}^{d}})}.
$$
Hence we have
$$
\|u\|_{Y}\lesssim 1+\tau^{1-\frac{(d-2s)(p-1)}{4s}}\|u\|_{Y}^{p-1}.
$$
As $2\leq p\leq 1+\frac{4s}{d-2s}$, the desired result holds by standard
continuing argument if $\tau$ is sufficiently small.
\end{proof}

Now we follow the strategy in \cite{Tao2} to establish a scattering criterion for radial solution to \eqref{equ:equ1.1}.

\begin{lemma}\label{scacri}
	Assume that $u$ is a radial solution to (1.1) satisfying
	$$
	\|u\|_{L_{t}^{\infty}H_{x}^{s}}\leq E.
	$$
	Suppose that there exist $\varepsilon=\varepsilon_E>0$ and $R=R_E>0$ such that if
	\begin{equation}
	\label{equ:masslim}
	\liminf_{t\rightarrow\infty}\displaystyle\int_{|x|\leq R}|u(t,x)|^{2}dx
	\leq\varepsilon^{2},
	\end{equation}
	then $u$ scatters forward in time.
\end{lemma}

\begin{proof}
	Suppose that $0<\varepsilon<1$ and $R>1$ which will be chosen later. Note that the admissible pair $(q,r)=\left(2(p-1),\frac{d(p-1)}{s}\right)\in \Lambda_{d}(s_c)$, and it will be sufficient to prove the scattering once we show that
$$  \|u\|
_{L_{t}^{2(p-1)}L_{x}^{\frac{d(p-1)}{s}}([T,\infty)\times{{\bf R}^{d}})}<\infty
$$
for some large $T>0$.

We then claim that this estimate can be reduced to
$$ \|e^{-i(t-T)(-\Delta)^{s}}u(T)\|_{L_t^{2(p-1)}L_x^{\frac{d(p-1)}{s}}([T,\infty)\times\mathbb{R}^d)}<\epsilon
$$
for $T>0$ large enough.

Indeed, this is a consequence of the Strichartz inequality, Sobolev embedding, fractional chain rule and, H\"{o}lder inequality, and continuity argument. The key observation is that for any interval $I=[t_1,t_2]$,
\begin{equation*}
\begin{split}
&\left\|\int_{t_1}^{t}e^{-i(t-t')(-\Delta)^s}(|u|^{p-1}u)(t')dt'\right\|_{L_t^{2(p-1)}L_x^{\frac{d(p-1)}{s}}(I\times\mathbb{R}^d)} \\ &\lesssim
\||D_x|^{s_c}(|u|^{p-1}u)\|_{L_t^{2(p-1)}L_x^{\frac{2d}{d+2s}}(I\times\mathbb{R}^d)}\\
&\lesssim \|u\|_{L_t^{2(p-1)}L_x^{\frac{d(p-1)}{s}}(I\times\mathbb{R}^d)}\|u\|_{L_t^{\infty}\dot{H}^{s_c}(I\times\mathbb{R}^d)}.
\end{split}
\end{equation*}

 First note that Sobolev embedding, Strichartz estimates, monotone convergence theorem yield
\begin{equation}
\label{2.28}
\|e^{-it(-\Delta)^{s}}u_{0}\|
_{L_{t}^{2(p-1)}L_{x}^{\frac{d(p-1)}{s}}([T_{0},\infty)\times{{\bf R}^{d}})}
<\varepsilon.
\end{equation}
for $T_0$ sufficiently large.

Using the non-concentration assumption \eqref{equ:masslim} and choosing $T>T_{0}>\varepsilon^{-\frac{2s(p+1)-d(p-1)}{2s}}$,
we have
\begin{equation}
\label{2.29}
\displaystyle\int_{{\bf R}^{d}}\chi_{R}(x)|u(T,x)|^{2}dx\leq\varepsilon^{2},
\end{equation}
where $\chi_R(x)=\chi(R^{-1}x)$, with $\chi\in C_c^{\infty}(\mathbb{R}^3)$, radial and $\chi\equiv 1$ when $|x|\leq 1$.

Now we write
\begin{equation*}
\begin{split}
&e^{-i(t-T)(-\Delta)^{s}}u(T)=e^{-it(-\Delta)^{s}}u_0+F_1(t)+F_2(t),\\
&F_j(t)=i\int_{I_j}e^{-i(t-t')(-\Delta)^{s}}(|u|^{p-1}u)(t')dt',\;j=1,2,
\end{split}
\end{equation*}
$\textrm{with }I_1=[T-\epsilon^{-\alpha},T],I_2=[0,T-\epsilon^{-\alpha}],\textrm{ with a parameter }\alpha \textrm{ to be chosen later}.
$

To ensure the smallness of the term $e^{-i(t-T)(-\Delta)^{s}}u(T)$, it remains to estimate $F_1$ and $F_2$ separately. $F_1$ represents the part which evolutes long enough, and it will be reasonable to apply \eqref{equ:masslim}. While for the short time part $F_2$, we will utilize the dispersive estimate \eqref{usualdispersive}.

{\bf Estimate of $F_1$: }

Utilizing Strichartz estimates, H\"{o}lder inequality, \eqref{localtime1}, \eqref{localtime2} and the fact that $s>s_c$
gives
\begin{equation}\label{F1-1}
\begin{split}
\|F_{1}\|_{L_{t}^{2(p-1)}L_{x}^{\frac{d(p-1)}{s}}([T,\infty)\times {\bf R}^{d})}
&\lesssim \displaystyle\int_{I_{1}}\||u|^{p-1}u\|_{H_{x}^{s}}ds \nonumber\\
&\lesssim \|u\|_{L_{t}^{2}W_{x}^{s,\frac{2d}{d-2s}}}
\|u\|_{L_{t}^{p-1}L_{x}^{\infty}}^{\frac{p-1}{2}}
\|u\|_{L_{t}^{\infty}L_{x}^{\frac{d(p-1)}{2s}}}^{\frac{p-1}{2}}\\
&\lesssim  |I_{1}|\|u\|_{L_t^{\infty}L_x^{\frac{d(p-1)}{2s}}(I_1\times\mathbb{R}^d)}^{\frac{p-1}{2}},
\end{split}
\end{equation}
where the three space-time norms in the second line are over $I_{1}\times {\bf R}^{d}$. We can choose $\alpha<\frac{2s(p+1)-d(p-1)}{8s}$, thanks to $p<1+\frac{4s}{d-2s}$.

In order to control $\|u\|_{L_t^{\infty}L_x^{\frac{d(p-1)}{2s}}(I_1\times\mathbb{R}^d)}$, we need estimate the time-derivative
\begin{equation*}
\begin{split}
&\left|\partial_t\int \chi_R(x)|u(t,x)|^2dx\right|\\=&2\Im\int \chi_R(x)\ov{u}(-\Delta)^s udx\\
=&2\Im \int ((-\Delta)^{s/2}(\chi_R)\ov{u})
(-\Delta)^{s/2}udx\\
+&2\Im \int \left((-\Delta)^{s/2}(\chi_R \ov{u})-\chi_R (-\Delta)^{s/2}\ov{u}-\ov{u}(-\Delta)^{s/2}\chi_R\right)\cdot (-\Delta)^{s/2}udx.
\end{split}
\end{equation*}

Applying Lemma \ref{KPV} and Sobolev imbedding with $p=2,p_1=\frac{6}{s}$, we finally have
$$ \left|\partial_t\int \chi_R(x)|u(t,x)|^2dx\right|\lesssim R^{-s/2}.
$$

Therefore,
$$
\|\chi_Ru\|_{L_{I_1}^{\infty}L_x^2}\lesssim \epsilon +R^{-\frac{s}{2}}\epsilon^{-\alpha}\lesssim \epsilon,
$$
by taking $R\gg \epsilon^{-\frac{4\alpha}{s}}$.

Using interpolation, Sobolev embedding, Lemma \ref{Strauss} and choosing $R$ such that
$R^{-\frac{2s(d-2s)}{d(p-1)}}\ll \epsilon^{\frac{2s(p+1)-d(p-1)}{2s(p-1)}}$,
we find that
\begin{eqnarray}
\|u\|_{L_{t}^{\infty}L_{x}^{\frac{d(p-1)}{2s}}}
&\lesssim& \|\chi_{R}u\|_{L_{t}^{\infty}L_{x}^{2}}^{1-\theta}
\|u\|_{L_{t}^{\infty}L_{x}^{\frac{2d}{d-2s}}}^{\theta}
+\|(1-\chi_{R})u\|_{L_{t,x}^{\infty}}^{1-\beta}
\|u\|_{L_{t}^{\infty}L_{x}^{2}}^{\beta} \nonumber\\
&\lesssim& \varepsilon^{\frac{2s(p+1)-d(p-1)}{2s(p-1)}}
+R^{-\frac{2s(d-2s)}{d(p-1)}}
\lesssim \varepsilon^{\frac{2s(p+1)-d(p-1)}{2s(p-1)}}
\end{eqnarray}
where $\theta=\frac{d(p-1)-4s}{2s(p-1)}$, $\beta=1-\frac{4s}{d(p-1)}$
and all space-time norms are over $I_{1}\times {\bf R}^{N}$.

In summary, we have obtained that for some $\gamma>0$, we have
\begin{equation}\label{F1-2}
\|F_{1}\|_{L_{t}^{2(p-1)}L_{x}^{\frac{d(p-1)}{s}}([T,\infty)\times {\bf R}^{d})} \lesssim \epsilon^{\gamma}.
\end{equation}

{\bf Estimate of $F_2$: }

Now we turn to $F_{2}$ and use interpolation to get
\begin{equation}\label{F2interpolation1}
\|F_{2}\|_{L_{t}^{2(p-1)}L_{x}^{\frac{d(p-1)}{s}}([T,\infty)\times {\bf R}^{d})}
\lesssim
\|F_{2}\|_{L_{t}^{2(p-1)}L_{x}^{\frac{2d(p-1)}{d(p-1)-2s}}([T,\infty)\times {\bf R}^{d})}^{\theta}
\|F_{2}\|_{L_{t}^{4\sigma}L_{x}^{\frac{2}{\delta}}([T,\infty)\times {\bf R}^{d})}^{1-\theta},
\end{equation}
where $\theta=\frac{2s-d(p-1)\delta}{d(p-1)(1-\delta)-2s}$ and $\delta \ll 1$.

As
$$
 F_2(t)=e^{-i(t-T+\epsilon^{-\alpha})(-\Delta)^s}(u(T-\epsilon^{-\alpha})-u_0).
$$
and $\left(2(p-1),\frac{2d(p-1)}{d(p-1)-2s}\right)\in \Lambda_{d}(0)$, we can use Strichartz estimates to get
\begin{equation}\label{F2interpolation2}
\|F_{2}\|_{L_{t}^{2(p-1)}L_{x}^{\frac{2d(p-1)}{d(p-1)-2s}}([T,\infty)\times {\bf R}^{d})}\lesssim 1.
\end{equation}
From the dispersive estimate \eqref{usualdispersive} and Leibnitz rule, we have
\begin{equation}\label{F2-1}
\begin{split}
\|F_{2}\|_{L_{x}^{\frac{2}{\delta}}}
&\lesssim  \displaystyle\int_{I_{2}}|t-s|^{-\frac{d(1-\delta)}{2}}
\||(-\Delta)^{\frac{d(1-s)(1-\delta)}{2}}(u|^{p-1}u)\|_{L_{x}^{\frac{2}{2-\delta}}}ds \\
&\lesssim  \displaystyle\int_{I_{2}}|t-s|^{-\frac{d(1-\delta)}{2}}
\|u\|_{L_{x}^{(p-1)\tilde{p}}}^{p-1}
\||(-\Delta)^{\frac{d(1-s)(1-\delta)}{2}}u\|_{L_{x}^{\tilde{q}}}ds,
\end{split}
\end{equation}
where $\tilde{q}=\frac{2d}{d-2s+2d(1-s)(1-\delta)}$,
$\tilde{p}=\frac{2d}{d+2s-2d(1-s)(1-\delta)-d\delta}$ when $d=3$ and
$\tilde{q}=2$, $\tilde{p}=\frac{2}{1-\delta}$ when $d\geq 4$. As $\frac{d}{d+1}\leq s<1$,
and $\frac{8s-3}{6}<\frac{p-1}{2}<\frac{2s}{3-2s}$ when $d=3$
and $\frac{1}{2}\leq\frac{p-1}{2}<\frac{2s}{N-2s}$, by Sobolev embedding, we find
$$
\|F_{2}\|_{L_{x}^{\frac{2}{\delta}}}
\lesssim \displaystyle\int_{I_{2}}|t-s|^{-\frac{d(1-\delta)}{2}}\|u\|_{H_{x}^{s}}^{p}ds
\lesssim (t-T+\epsilon^{-\alpha})^{-\frac{d(1-\delta)}{2}+1}.
$$
Therefore,
$$
\|F_{2}\|_{L_{t}^{2(p-1)}L_{x}^{\frac{2}{\delta}}([T,\infty)\times {\bf R}^{d})}
\lesssim \varepsilon^{\gamma_1},
$$
for some $\gamma_1>0$.
Combing with \eqref{F2interpolation1}and \eqref{F2interpolation2} gives
\begin{equation}\label{F2-2}
\|F_{2}\|_{L_{t}^{2(p-1)}L_{x}^{\frac{d(p-1)}{s}}([T,\infty)\times {\bf R}^{d})}
\lesssim \varepsilon^{c(d,p,\gamma_1)}
\end{equation}
with $c(d,p,\gamma_1)>0$.

Putting \eqref{2.28}, \eqref{F1-2}and \eqref{F2-2} together yields
$$
\|e^{i(t-T)(-\Delta)^{s}}u(T)\|
_{L_{t}^{2(p-1)}L_{x}^{\frac{d(p-1)}{s}}([T,\infty)\times {\bf R}^{N})}
\lesssim \epsilon^{c(d,p,\gamma,\gamma_1)}
$$
with $c(d,p,\gamma,\gamma_1)>0$.

This completes the proof.

\end{proof}

\begin{remark}
	The criterion is not sufficient to ensure scattering in the mass-critical case. We take mass critical NLS for an example, namely $s=1$ and $p=1+\frac4d$ in \eqref{equ:equ1.1}. For any $\epsilon>0,R>0$, we take
	$\lambda>0$ small enough and consider a solitary wave solution
	$$ u_{\lambda}(t,x)=\lambda^{\frac{d}{2}}e^{it}Q(\lambda x).
	$$
	One easily check that
	$$ \|u_{\lambda}\|_{L_{x}^{\infty}H_x^1}\leq C,\textrm{ independent of $\lambda$},
	$$
	and
	$$ \int_{|x|\leq R}|u_{\lambda}(t,x)|^2dx=\int_{|x|\leq\lambda R}|Q|^2dx\rightarrow 0,\lambda\rightarrow 0.	$$
	However, for any $\lambda>0$, $u_{\lambda}(t,x)$ does not scatter.
\end{remark}


\section{Morawetz Estimates}

In this section, we will use the virial identity established in \cite{Lenzmann} to prove Morawetz estimate. The key ingredient is to use the Balakrishnan's formula for the fractional Laplacian $(-\Delta)^s$ for $s\in (0,1)$:
\begin{equation}\label{Bformula}
(-\Delta)^s=\frac{\sin(\pi s)}{\pi}\int_0^{\infty}\lambda^{s-1}(-\Delta)(-\Delta+\lambda)^{-1}d\lambda.
\end{equation}

 From Plancherel, one easily checks that
\begin{equation}\label{hsnorm}
 s\|(-\Delta)^{s/2}u\|_{L^2(\mathbb{R}^d)}^2=\int_0^{\infty}\lambda^sd\lambda\int_{\mathbb{R}^d}|\nabla u_{\lambda}|^2dx
\end{equation}
with the notation $u_{\lambda}=\sqrt{\frac{\sin(\pi s)}{\pi}}(\lambda-\Delta)^{-1}u$ here and in the sequel.
We remark that throughout this section, we work for any dimension $d$ and any nonlinearity so that $s_c=\frac{d}{2}-\frac{2s}{p-1}\in (0,s)$.

\subsection{Virial Identity}
The key ingredient is the following virial identity: For reasonable radial function $\varphi=\varphi(x)$, we have
\begin{lemma}[\cite{Lenzmann}]\label{virial}
	Suppose $u$ solves \eqref{equ:equ1.1}. Then for $$M_{\varphi}(t)=2\Im\int_{\mathbb{R}^d} \ov{u}\nabla u\cdot\nabla \varphi(x)dx,$$ we have
	\begin{equation}\label{virial1}
	\begin{split}
	\frac{dM_{\varphi}}{dt}=&\int_{0}^{\infty}\lambda^{s}d\lambda\int_{\mathbb{R}^d} \left(4\partial_j\partial_k \varphi\partial_j\ov{u_{\lambda}}\partial_ku_{\lambda}-\Delta^2\varphi |u_{\lambda}|^2\right)dx\\
	&-\frac{2(p-1)}{p+1}\int_{\mathbb{R}^d}\Delta\varphi |u|^{p+1}dx,
	\end{split}
	\end{equation}
	where $u_{\lambda}=\sqrt{\frac{\sin(\pi s)}{\pi}}(\lambda-\Delta)^{-1}u$.
\end{lemma}
For the convenience of the reader, we will review the proof in appendix.

Next we prove some commutator type estimates, which will be useful in the following subsections. We acknowledge that these types of proofs are similar as in \cite{Lenzmann}, with a different manner to fit our need.
\begin{lemma}\label{commutator}
	For $s>\frac{1}{2}$, we have
	$$ \left|\int_0^{\infty}\lambda^sd\lambda\int_{\mathbb{R}^d} \chi_R\Delta \chi_R |u_{\lambda}|^2dx\right|\lesssim \frac{1}{R^{2s}}.
	$$
\end{lemma}
\begin{proof}
	We write
	$$ \left(\int_0^{M}+\int_M^{\infty}\right)\lambda^sd\lambda \int_{\mathbb{R}^d}\chi_R\Delta\chi_R|u_{\lambda}|^2dx:=I+II,
	$$
	with $M>0$ to be chosen later.
	\begin{equation*}
	\begin{split}
	|I|&\leq \int_0^M\lambda^sd\lambda\|\chi_R\|_{L_x^d}\|\Delta\chi_R\|_{L_x^{\infty}}\|u_{\lambda}\|_{L_x^2}\|u_{\lambda}\|_{L_x^{\frac{2d}{d-2}}}\\
	&\lesssim \frac{1}{R}\int_0^M\lambda^sd\lambda^{s-\frac{3}{2}}d\lambda \\
	&\lesssim M^{s-\frac{1}{2}}R^{-1},
	\end{split}
	\end{equation*}
	thanks to $s>\frac{1}{2}$, where in the final step, we have used Sobolev imbedding and the estimate
	$$ \|u_{\lambda}\|_{L_x^2}\lesssim \lambda^{-1}\|u\|_{L_x^2},
	\|\nabla u_{\lambda}\|_{L_x^2}\lesssim \lambda^{-\frac{1}{2}}\|u\|_{L_x^2}.
	$$
	\begin{equation*}
	\begin{split}
	|II|&\lesssim \frac{1}{R^2}\int_M^{\infty}\lambda^sd\lambda\int_{\mathbb{R}^d}|u_{\lambda}|^2d\lambda\\
	&\lesssim \frac{1}{R^2}\int_M^{\infty}\lambda^sd\lambda\int_{\mathbb{R}^d}\frac{|\widehat{u}(\xi)|^2}{(|\xi|^2+\lambda)^2}d\xi\\
	&\lesssim \frac{1}{R^2M^{1-s}}.
	\end{split}
	\end{equation*}
	We optimally choose $M=R^{-2}$, thus
	$$ \left|\int_0^{\infty}\lambda^sd\lambda \int_{\mathbb{R}^d}\chi_R\Delta\chi_R|u_{\lambda}|^2dx\right|\lesssim R^{-2s}.
	$$
\end{proof}
\begin{lemma}\label{commutator2}
$$ L_R:=\left|\int_0^{\infty}\lambda^sd\lambda\int_{\mathbb{R}^d}\left(\nabla
(-\Delta+\lambda)^{-1}[-\Delta,\chi_R]u_{\lambda}\right)\cdot \nabla (\chi_R\ov{u})_{\lambda}dx\right|\lesssim \frac{1}{R^{2s\theta}},
$$
for some $\theta\in\left(
\frac{1}{2s},1\right)$.
\end{lemma}
\begin{proof}
Pick $\theta\in(0,1)$ such that $1<2s\theta<2$ and $2s(1-\theta)<s$ .

Writing
$$ [-\Delta,\chi_R]u_{\lambda}=-(\Delta\chi_R)u_{\lambda}-2\nabla\chi_R\cdot\nabla u_{\lambda}
$$
and utilizing $\|\nabla (-\Delta+\lambda)^{-1}\|_{L^2\rightarrow L^2}\leq \lambda^{-1/2}$,  we have from Cauchy-Schwartz that
\begin{equation*}
\begin{split}
L_R&\leq\left(\int_0^{\infty}\lambda^{2s\theta-1}\left\|(\Delta\chi_R)u_{\lambda}+2\nabla\chi_R\cdot\nabla u_{\lambda}\right\|_{L_x^2}^2 d\lambda\right)^{1/2}
\left(
\int_{0}^{\infty}\lambda^{2s(1-\theta)} \|\nabla (\chi_R\ov{u})_{\lambda}\|_{L_x^2}^2d\lambda\right)^{1/2}\\
&\lesssim \|\chi_R\ov{u}\|_{\dot{H}_x^{2s(1-\theta)}}\left(\int_0^{\infty}\lambda^{2s\theta-1}\left\|(\Delta\chi_R)u_{\lambda}+2\nabla\chi_R\cdot\nabla u_{\lambda}\right\|_{L_x^2}^2 d\lambda\right)^{1/2}\\
&\lesssim
\left(\int_0^{\infty}\lambda^{2s\theta-1}\left\|(\Delta\chi_R)u_{\lambda}+2\nabla\chi_R\cdot\nabla u_{\lambda}\right\|_{L_x^2}^2 d\lambda\right)^{1/2}.
\end{split}
\end{equation*}

Taking $M>0$ to be chosen later, we have
\begin{equation*}
\begin{split}
\int_0^M\lambda^{2s\theta-1}d\lambda\int |\Delta\chi_R|^2|u_{\lambda}|^2dx
&\lesssim \int_{0}^{M}\lambda^{2s\theta-1}\|\Delta\chi_R\|_{L_x^d}^2\|u_{\lambda}\|_{L_x^{\frac{2d}{d-2}}}^2d\lambda\\
&\lesssim \frac{1}{R^2}\int_0^{M}\lambda^{2s\theta-2}d\lambda\\
&\lesssim \frac{M^{2s\theta-1}}{R^2}.
\end{split}
\end{equation*}

\begin{equation*}
\begin{split}
\int_M^{\infty}\lambda^{2s\theta-1}d\lambda\int |\Delta\chi_R|^2|u_{\lambda}|^2dx&\lesssim \int_M^{\infty}\lambda^{2s\theta-1}\|\Delta\chi_R\|_{L_x^{\infty}}^2\|u_{\lambda}\|_{L_x^2}^2   d\lambda\\
&\lesssim \frac{1}{R^4}\int_M^{\infty}\lambda^{2s\theta-3}d\lambda\\
&\lesssim \frac{1}{R^4M^{2-2s\theta}}.
\end{split}
\end{equation*}
We optimally choose $M=R^{-2}$, thus
$$ \int_0^{\infty}\lambda^{2s\theta-1}d\lambda \int |\Delta\chi_R|^2|u_{\lambda}|^2dx\lesssim \frac{1}{R^{4s\theta}}.
$$
\end{proof}

\begin{lemma}\label{Mbound}
Suppose $\nabla\varphi\in  W^{2,\infty}(\mathbb{R}^d)$ and $u\in H^{\frac{1}{2}}(\mathbb{R}^d)$, then
$$ \left|\int_{\mathbb{R}^d}\ov{u}\nabla u\cdot\nabla \varphi dx\right|\lesssim \|u\|_{H^{\frac{1}{2}}(\mathbb{R}^d)}^2\|\nabla \varphi\|_{W^{2,\infty}(\mathbb{R}^d)}.
$$
\end{lemma}
This estimate appears in the appendix of \cite{Lenzmann} with a weaker condition $\nabla\varphi\in W^{1,\infty}(\mathbb{R}^d)$. Here we perform a different proof. Though we need a stronger condition on $\nabla\varphi,$ our proof has its own interest where we apply formula \eqref{hsnorm} instead of physical space characterization of fractional Sobolev space.
\begin{proof}
By duality, it is reduced to estimate
$\|D^{1/2}(\ov{u}\nabla\varphi)\|_{L_x^2}$.

Applying \eqref{hsnorm}, we write
\begin{equation*}
\begin{split}
\|D^{1/2}(\ov{u}\nabla\varphi)\|_{L_x^2}^2=& 2\left(\int_0^1+\int_1^{\infty}\right)\lambda^{1/2}d\lambda\int |\nabla (\ov{u}\nabla\varphi)_{\lambda}|^2dx\\
=&:I+II.
\end{split}
\end{equation*}
For the first term,
\begin{equation*}
\begin{split}
I&\lesssim \int_0^1\lambda^{1/2}(\lambda^{-1/2}\|\ov{u}\nabla\varphi\|_{L_x^2})^2d\lambda\\
&\lesssim \|u\|_{L_x^2}^2\|\nabla\varphi\|_{L_x^{\infty}}^2.
\end{split}
\end{equation*}
While for the second term
\begin{equation*}
\begin{split}
II&\lesssim \int_1^{\infty}\lambda^{1/2}\left(\|(\nabla\ov{u}\otimes\nabla\varphi)_{\lambda}\|_{L_x^2}^2+\|(\ov{u}\nabla^2\varphi)_{\lambda}\|_{L_x^2}^2\right)d\lambda,
\end{split}
\end{equation*}
Denote $R(\lambda)=(-\Delta+\lambda)^{-1}$, and write
\begin{equation*}
\begin{split}
(\nabla\ov{u}\otimes\nabla\varphi)_{\lambda}&=\nabla\varphi\otimes\nabla\ov{u_{\lambda}}+[R(\lambda),\nabla\varphi]\nabla\ov{u}\\&=\nabla\varphi\otimes\nabla\ov{u_{\lambda}}
+R(\lambda)[\nabla\varphi,-\Delta]\nabla\ov{u_{\lambda}}.
\end{split}
\end{equation*}
We estimate
$$ \|\ov{u}\nabla^2\varphi\|_{L_x^2}\lesssim \lambda^{-1}\|\nabla^2\varphi\|_{L_x^{\infty}}\|u\|_{L_x^2},
\|\nabla\varphi\otimes\nabla\ov{u_{\lambda}}\|_{L_x^2}\lesssim \|\nabla\varphi\|_{L_x^{\infty}}\|\nabla\ov{u_{\lambda}}\|_{L_x^2},
$$
and
$$\|R(\lambda)[\nabla\varphi,-\Delta]\nabla\ov{u_{\lambda}}\|_{L_x^2}\lesssim \lambda^{-1}\left(\|\nabla\Delta\varphi\|_{L_x^{\infty}}\|\nabla\ov{u_{\lambda}}\|_{L_x^2}+
\|\nabla^2\varphi\|_{L_x^{\infty}}\|\ov{u_{\lambda}}\|_{L_x^2}\right),
$$
and these imply that
$$ II\lesssim \|\nabla\varphi\|_{W_x^{2,\infty}}^2\|u\|_{H_x^{\frac{1}{2}}}^2.
$$
\end{proof}
\subsection{Coercivity}
We review the variational analysis related to the ground state $Q$. See \cite{Lenzmann}.
Recall that $Q$ is a radial solution to
\begin{equation}\label{elliptic}
(-\Delta)^sQ+Q-Q^{p}=0,
\end{equation}
and $Q$ is an extremum of the sharp Gagliardo-Nirenberg inequality, namely
$$\|Q\|_{L^{p+1}(\mathbb{R}^d)}^{p+1}= C(d,p,s)\|(-\Delta)^{s/2}Q\|_{L^2(\mathbb{R}^d)}^\frac{4s_c}{d-2s_c}\|Q\|_{L^2(\mathbb{R}^d)}^{p+1-\frac{4(s-s_c)}{d-2s_c}}.
$$
Multiplying by $x\cdot \nabla Q$ to the equation \eqref{elliptic}, and applying the following Pohozaev identity on $Q$
$$\int_{\mathbb{R}^d}x\cdot \nabla Q(x)dx=-\frac{d-2s}{2}\|Q\|_{\dot{H}_x^s}^2,
$$
we can obtain (see also the appendix in \cite{Lenzmann})
\begin{proposition}\label{soliton}
$$\|Q\|_{\dot{H}_x^s}^{s_c}\|Q\|_{L_x^2}^{s-s_c}=\left(\frac{2d+4(s-s_c)}{2dC(d,p,s)}\right)^{\frac{d-2s_c}{4}}.
$$
\end{proposition}

Now we prove a coercivity result:
\begin{lemma}\label{coercivity1}
Suppose
\begin{equation}\label{threshold}
\begin{split}
& E[u_0]^{s_c}M[u_0]^{s-s_c}<(1-\delta)E[Q]^{s_c}M[Q]^{s-s_c},\\
& \|(-\Delta)^{s/2}u_0\|_{L^2(\mathbb{R}^d)}^{s}\|u_0\|_{L^2(\mathbb{R}^d)}^{s-s_c}\leq
\|(-\Delta)^{s/2}Q\|_{L^2(\mathbb{R}^3)}^{s_c}\|Q\|_{L^2(\mathbb{R}^d)}^{s-s_c}.
\end{split}
\end{equation}

Then there exists $\delta'$, depending only on $\delta,$ such that in the life span (i.e. $|t|<T^*$), the flow $u(t)$ of \eqref{equ:equ1.1} satisfies
$$ \|(-\Delta)^{s/2}u(t)\|_{L^2(\mathbb{R}^d)}^{s_c}\|u(t)\|_{L^2(\mathbb{R}^d)}^{s-s_c}\leq  (1-\delta')
\|(-\Delta)^{s/2}Q\|_{L^2(\mathbb{R}^d)}^{s_c}\|Q\|_{L^2(\mathbb{R}^d)}^{s-s_c}.
$$
\end{lemma}
\begin{proof}
	From the conservation of mass and energy, we have
\begin{equation*}
\begin{split}
(1-\delta)^{\frac{1}{s_c}}E[Q]M[Q]^{\frac{s-s_c}{s_c}}&>E[u_0]M[u_0]^{\frac{s-s_c}{s_c}}\\&=
\left(\frac{1}{2}\|u(t)\|_{\dot{H}_x^s}^2-\frac{1}{p+1}\|u(t)\|_{L_x^{p+1}}^{p+1}\right)\|u(t)\|_{L_x^2}^{\frac{2(s-s_c)}{s_c}}\\
&\geq
\frac{1}{2}\|u(t)\|_{\dot{H}_x^s}^2\|u(t)\|_{L_x^2}^{\frac{2(s-s_c)}{s_c}}-\frac{C_{d,p,s}}{p+1}\|u(t)\|_{\dot{H}_x^s}^{\frac{2d}{d-2s_c}}
\|u(t)\|_{L_x^2}^{\frac{4(s-s_c)}{d-2s_c}+\frac{2(s-s_c)}{s_c}},
\end{split}
\end{equation*}
where we have used the sharp Gagliardo-Nirenberg inequality in the last step. Since $Q$ is the extremum of sharp Gagliardo-Nirenberg, we have
$$ 1-\delta\geq \left(\frac{d}{s_c}\right)^{s_c}\left(\frac{1}{2}y(t)^2-\frac{2s}{d(p-1)}y(t)^{\frac{2d}{d-2s_c}}\right)
$$
with
$$ y(t)=\frac{\|u(t)\|_{\dot{H}_x^s}\|u(t)\|_{L_x^2}^{\frac{s-s_c}{s_c}}}{\|Q\|_{\dot{H}_x^s}\|Q\|_{L_x^2}^{\frac{s-s_c}{s_c}}}.
$$
Notice that the initial condition gives
$$ 1-\delta\geq y(0)^{2s_c},
$$
and the lemma follows from continuity argument.
\end{proof}

\begin{remark}
	From the blow up criterion for energy sub-critical FNLS: 
$$ T^*<\infty \Longrightarrow
 \limsup_{t\rightarrow T^*}\|u(t,\cdot)\|_{H_x^s}=\infty,
$$
this lemma implies global well-posedness of FNLS under condition \eqref{threshold} with radial initial data in $H^s(\mathbb{R}^d)$.
\end{remark}

Now pick $\chi_R(x)=\chi(R^{-1}x)$ with a radial function $\chi\in C_c^{\infty}(\mathbb{R}^d)$, $\chi(x)=1,|x|\leq \frac{1}{2}$ and $0\leq \chi \leq 1$.

\begin{lemma}\label{coercivity2}
Under the assumption of lemma \ref{coercivity1}, there exists $R>0$, large enough and $\tilde{\delta}>0$, so that
$$ s\|\chi_R u(t)\|_{\dot{H}_x^s}^2-\frac{d(p-1)}{2(p+1)}\|\chi_R u(t)\|_{L_x^{p+1}}^{p+1},\geq \tilde{\delta} \|\chi_R u(t)\|_{L_x^{p+1}}^{p+1},\forall t\in\mathbb{R}.
$$
\end{lemma}
\begin{proof}
This is another application of sharp Gagliardo-Nirenberg inequality. For any $v\in \dot{H}_x^s$, with
$$ \|v\|_{\dot{H}_x^s}^{s_c}\|v\|_{L_x^2}^{s-s_c}\leq (1-\delta')\|Q\|_{\dot{H}_x^s}^{s_c}\|Q\|_{L_x^2}^{s-s_c},
$$
we have
\begin{equation*}
\|v\|_{\dot{H}_x^s}^2-\frac{d(p-1)}{2s(p+1)}\|v\|_{L_x^{p+1}}^{p+1}=\frac{2d}{d-2s_c}E[v]-\left(\frac{d}{d-2s_c}-1\right)\|v\|_{\dot{H}_x^{s}}^{2}.
\end{equation*}
From Sharp Gagliardo-Nirenberg,
\begin{equation*}
\begin{split}
E[v]&\geq \frac{1}{2}\|v\|_{\dot{H}_x^s}^2-\frac{C(d,p,s)}{p+1}\|v\|_{\dot{H}_x^s}^{\frac{2d}{d-2s_c}}\|v\|_{L_x^{2}}^{\frac{4(s-s_c)}{d-2s_c}}\\
&\geq \|v\|_{\dot{H}_x^s}^2\left(\frac{1}{2}-\frac{C(d,p,s)}{p+1}(1-\delta')^{\frac{4}{d-2s_c}}
\|Q\|_{\dot{H}_x^s}^{\frac{4s_c}{d-2s_c}}\|Q\|_{L_x^2}^{\frac{4(s-s_c)}{d-2s_c}}\right)\\
&=\|v\|_{\dot{H}_x^s}^2\left(\frac{1}{2}-(1-\delta')^{\frac{4}{d-2s_c}}\frac{d-2s_c}{2d}\right),
\end{split}
\end{equation*}
where we have used Proposition \ref{soliton} in the last step. Thus
\begin{equation*}\label{coer}
\begin{split}
 &\|v\|_{\dot{H}_x^s}^2-\frac{d(p-1)}{2s(p+1)}\|v\|_{L_x^{p+1}}^{p+1}
\geq \delta''\|v\|_{\dot{H}_x^s}^2,\\
&\|v\|_{\dot{H}_x^s}^2-\frac{d(p-1)}{2s(p+1)}\|v\|_{L_x^{p+1}}^{p+1}\geq \frac{\delta''}{1+\delta''}\frac{2d(p-1)}{p+1}\|v\|_{L_x^{p+1}}^{p+1}.
\end{split}
\end{equation*}
To finish the proof, we just need verify the inequality
$$ \|\chi_Ru(t)\|_{\dot{H}_x^s}^{s_c}\|\chi_Ru(t)\|_{L^2}^{s-s_c}\leq (1-\delta''')\|Q\|_{\dot{H}_x^s}^{s_c}\|Q\|_{L^2}^{s-s_c}.
$$
Indeed, from integration by parts,
\begin{equation*}
\begin{split}
\int |\nabla (\chi_R v)|^2dx=&\int \chi_R^2|\nabla v|^2dx+
\int |\nabla \chi_R|^2|v|^2\\&+2\Re\int \chi_R\nabla \chi_R\cdot \ov{v}\nabla v dx
\\=&\int \chi_R^2|\nabla v|^2dx+
\int |\nabla \chi_R|^2|v|^2 +\frac{1}{2}\int \nabla(\chi_R^2)\cdot\nabla |v|^2dx\\
=&\int \chi_R^2|\nabla v|^2dx-\int \chi_R\Delta\chi_R |v|^2dx,
\end{split}
\end{equation*}
and
\begin{equation*}
\begin{split}
\nabla(\chi_Ru_{\lambda})&=\nabla(\chi_R u)_{\lambda}+\nabla[\chi_R,(-\Delta+\lambda)^{-1}]u\\
&=\nabla(\chi_R u)_{\lambda}+
(-\Delta+\lambda)[-\Delta+\lambda,\chi_R](-\Delta+\lambda)^{-1}u\\
&=\nabla(\chi_R u)_{\lambda}
+(-\Delta+\lambda)^{-1}[-\Delta,\chi_R]u_{\lambda}.
\end{split}
\end{equation*}
We write
\begin{equation*}
\begin{split}
s\|\chi_R u\|_{\dot{H}_x^s}^2&=
\int_0^{\infty}\lambda^sd\lambda\int \chi_R^2|\nabla  u_{\lambda}|^2dx+
O\left(\frac{1}{R^{2s\theta}}\right)\\
&\leq \|u\|_{\dot{H}_x^s}^2+O\left(\frac{1}{R^{2s\theta}}\right)
\end{split}
\end{equation*}
for some $\theta\in\left(\frac{1}{2s},1\right)$, thanks to Lemma \ref{commutator} and Lemma \ref{commutator2}.

 Finally we can choose $R=R(s,u_0,Q)$ large enough to complete the proof.
\end{proof}

\subsection{Morawetz Estimate}

\begin{proposition}[Morawetz estimate]\label{Morawetz}
Suppose $u$ is a global solution to \eqref{equ:equ1.1} with the condition \eqref{threshold}. Then there exists $\beta>0$,$C>0$, depending only on $s,\|u_0\|_{H_x^s},Q$, such that for any $t_1,t_2\in\mathbb{R}$,
\begin{equation}\label{Mora}
\frac{1}{|t_2-t_1|}\int_{t_1}^{t_2}\int_{|x|\leq \frac{R}{2}}|u(t,x)|^{p+1}dxdt\leq C\left(\frac{R}{|t_2-t_1|}+\frac{1}{R^{\beta}}\right).
\end{equation}
\end{proposition}
\begin{proof}
	For sufficiently large $R>0$, define the function
	$$ \varphi_R(x)=R^2\psi\left(\frac{|x|}{R}\right),
	$$
	where
	$$ \psi(r)=\frac{r^2}{2},0\leq r\leq 1,\psi(r)=r, r\geq 2.
	$$
Direct calculation gives
\begin{equation*}
\begin{split}
& \partial_{kl}^2\varphi_R(x)=R\frac{\delta_{kl}|x|^2-x_kx_l}{|x|^3}\psi'\left(\frac{|x|}{R}\right)+\frac{x_kx_l}{|x|^2}\psi''\left(\frac{|x|}{R}\right),\\
&\Delta\varphi_R(x)=\frac{(d-1)R}{|x|}\psi'\left(\frac{|x|}{R}\right)+\psi''\left(\frac{|x|}{R}\right),\\
\nabla\Delta\varphi_R(x)=&-\frac{(d-1)Rx}{|x|^3}\psi'\left(\frac{|x|}{R}\right)+\frac{(d-1)x}{|x|^2}\psi''\left(\frac{|x|}{R}\right)\\
&+\frac{x}{R|x|}\psi'''\left(\frac{|x|}{R}\right),\\
\Delta^2\varphi_R(x)=&\frac{(d-1)(d-3)}{|x|^2}\psi''\left(\frac{|x|}{R}\right)-\frac{R(d-1)(d-3)}{|x|^3}\psi'\left(\frac{|x|}{R}\right)\\
&+\frac{2(d-1)}{R|x|}\psi'''\left(\frac{|x|}{R}\right)+\frac{1}{R^2}\psi^{(4)}\left(\frac{|x|}{R}\right).
\end{split}
\end{equation*}
Note that Hess$\varphi_R$ is non-negative definite.

Pick $\chi_R(x)=\chi(R^{-1}x)$ with a radial function $\chi\in C_c^{\infty}(\mathbb{R}^d)$, $\chi(x)=1,|x|\leq \frac{1}{2}$ and $0\leq \chi \leq 1$.

 We have
\begin{equation}\label{positive}
\begin{split}
&\int_0^{\infty}\lambda^sd\lambda
\int_{\mathbb{R}^d}4\partial_{kl}^2\varphi_R\partial_k\ov{u_{\lambda}}\partial_lu_{\lambda} dx\\
\geq &4\int_0^{\infty}\lambda^sd\lambda
\int_{\mathbb{R}^d}\chi_R^2|\nabla u_{\lambda}|^2dx\\
=&4\int_0^{\infty}\lambda^sd\lambda
\int_{\mathbb{R}^d}|\nabla(\chi_{R}u_{\lambda})|^2dx+
4\int_0^{\infty}\lambda^sd\lambda
\int_{\mathbb{R}^d}\chi_R\Delta(\chi_R)|u_{\lambda}|^2dx,
\end{split}
\end{equation}
where in the last step, we have used the same integration by parts in the proof of Lemma \ref{coercivity2}.

Writing
\begin{equation*}
\begin{split}
\nabla(\chi_Ru_{\lambda})
&=\nabla(\chi_R u)_{\lambda}
+(-\Delta+\lambda)^{-1}[-\Delta,\chi_R]u_{\lambda},
\end{split}
\end{equation*}
It follows from Lemma \ref{coercivity2} that
\begin{equation}\label{virial2}
\begin{split}
\frac{dM_{\varphi_R}}{dt}&\geq \tilde{\delta}\|\chi_R u\|_{L_x^{p+1}}^{p+1} -C\int_{|x|\geq \frac{R}{2}}|u|^{p+1}dx\\&+4\int_{0}^{\infty}\lambda^sd\lambda
\int_{\mathbb{R}^d}\chi_R\Delta\chi_R |u_{\lambda}|^2dx-\int_{0}^{\infty}\lambda^sd\lambda
\int_{\mathbb{R}^d} \Delta^2\varphi_R|u_{\lambda}|^2dx \\&+2\Re\int_0^{\infty}\lambda^sd\lambda\int_{\mathbb{R}^d}\left(\nabla
(-\Delta+\lambda)^{-1}[-\Delta,\chi_R]u_{\lambda}\right)\cdot \nabla (\chi_R\ov{u})_{\lambda}dx
\end{split}
\end{equation}

There are several remaining terms to be estimated:

Firstly, useing Strauss inequality Lemma \ref{Strauss} yields that
\begin{equation*}
\int_{|x|\geq \frac{R}{2}}|u|^{p+1}dx\lesssim
\|u\|_{L_t^{\infty}L^{\infty}(|x|\geq \frac{R}{2})}^{p-1}\|u\|_{L_t^{\infty}L_x^2}^2
\lesssim R^{-\frac{(d-s)(p-1)}{2}}.
\end{equation*}
Next,
\begin{equation*}
\begin{split}
\int_0^{\infty}\lambda^sd\lambda\int \Delta^2\varphi_R |u_{\lambda}|^2dx&=
\int_0^{\infty}\lambda^sd\lambda\int \Delta\varphi_R 2\left((\Re \ov{u_{\lambda}}{\Delta u_{\lambda}})+\nabla|u_{\lambda}|^2\right)dx\\
&\lesssim R^{-2s}.
\end{split}
\end{equation*}
Here, to arrive at the last step, one just plays the same game as in the proofs of Lemma \ref{commutator}, Lemma \ref{commutator2}, where the estimations for the other terms are already contained therein.

Finally, since
$$ \|M_{\varphi_R}\|_{L_t^{\infty}}\lesssim R,
$$
thanks to Lemma \ref{Mbound},
we have for any  $t_1,t_2\in\mathbb{R}$,
$$ \int_{t_1}^{t_2}\|\chi_R u(t)\|_{L_x^{p+1}}^{p+1}dt\lesssim \frac{1}{R^{\beta}}|t_2-t_1|+R,
$$
for some $\beta>0$. This completes the proof.
\end{proof}
The following corollary is a simple consequence of the Morawetz estimate \eqref{Mora} and H\"{o}lder inequality.

\begin{corollary}
	There exist a sequence of time $\{t_n\}_{n\in\mathbb{N}}$ and a sequence of radii $\{R_n\}_{n\in\mathbb{N}}$, with $t_n\rightarrow\infty,R_n\rightarrow\infty$, such that
	$$ \lim_{t_n\rightarrow\infty}\int_{|x|\leq R_n}|u(t_n,x)|^2dx=0.
	$$
\end{corollary}

Combining this corollary and the scattering criterion (Lemma \ref{scacri}), the proof of Theorem \ref{scattering} is complete.

\section{Appendix}

\subsection{Virial Identity}
We briefly review the proof of Lemma \ref{virial} in \cite{Lenzmann}.

We first briefly explain the formula \eqref{Bformula}:

In the Fourier side, we compute
\begin{equation*}
\begin{split}
|\xi|^2\int_0^{\infty}\frac{\lambda^{s-1}}{\lambda+|\xi|^2}d\lambda&=
|\xi|^{2s}\int_0^{\infty}\frac{y^{s-1}}{1+y}dy\\
&=|\xi|^{2s}\int_1^{\infty}\frac{(y-1)^{s-1}}{y}dy\\
&=|\xi|^{2s}\int_0^1(1-t)^{s-1}t^{-s}dt\\
&=|\xi|^{2s}\Gamma(s)\Gamma(1-s)\\
&=|\xi|^2\frac{\pi}{\sin(s\pi)}
\end{split}
\end{equation*}

Now we derive the virial identity in a formal way. The rigorous derivation contains several steps of approximation accompanied with certain estimations.

Using the equation \eqref{equ:equ1.1} satisfied by $u$, we have
\begin{equation*}
\begin{split}
\frac{dM_{\varphi}}{dt}=&2\Im\int (i(-\Delta)^s\ov{u})\nabla u\cdot\nabla \varphi dx+
2\Im \int \ov{u}\nabla (-i(-\Delta)^s u)\cdot\nabla \varphi\\
&-\frac{2(p-1)}{p+1}\int_{\mathbb{R}^d}\Delta\varphi |u|^{p+1}dx.
\end{split}
\end{equation*}
Set $c_s=\frac{\sin(\pi s)}{\pi}, u_{\lambda}=\sqrt{c_s}(-\Delta+\lambda)^{-1}u,$ and then
\begin{equation*}
\begin{split}
&2\Im\int (i(-\Delta)^s\ov{u})\nabla u\cdot\nabla \varphi dx+
2\Im \int \ov{u}\nabla (-i(-\Delta)^s u)\cdot\nabla \varphi\\
=&2\sqrt{c_s}\Re\int_0^{\infty}\lambda^{s-1}d\lambda\int
\Delta\ov{u_{\lambda}}\nabla u\cdot\nabla \varphi dx+
2\sqrt{c_s}\Re\int_0^{\infty}\lambda^{s-1}d\lambda\int \ov{u}\nabla\varphi\cdot\nabla\Delta u_{\lambda}dx\\
=&2\sqrt{c_s}\Re\int_{0}^{\infty}\lambda^{s-1}d\lambda\int \nabla\ov{u_{\lambda}}\cdot\nabla(\nabla u\cdot\nabla \varphi)dx-2\sqrt{c_s}\Re\int_{0}^{\infty}\lambda^{s-1}d\lambda\int \Delta u_{\lambda}\nabla\ov{u}\cdot\nabla\varphi dx\\
&-2\sqrt{c_s}\Re\int_{0}^{\infty}\lambda^{s-1}d\lambda
\int \Delta u_{\lambda}\ov{u}\Delta\varphi dx.
\end{split}
\end{equation*}
Here
\begin{equation}\label{1}
\begin{split}
&\Re\int \Delta  u_{\lambda}\ov{u}\Delta\varphi dx\\
=&\frac{1}{\sqrt{c_s}}\Re\int \Delta u_{\lambda}(-\Delta+\lambda)\ov{u_{\lambda}}\Delta\varphi dx\\
=&-\frac{1}{\sqrt{c_s}}\Re\int |\Delta u_{\lambda}|^2\Delta\varphi dx +\frac{\lambda}{\sqrt{c_s}}\Re\int \Delta u_{\lambda}\ov{u_{\lambda}}\Delta\varphi dx\\
=&-\frac{1}{\sqrt{c_s}}\Re\int |\Delta u_{\lambda}|^2\Delta\varphi dx
+\frac{\lambda}{2\sqrt{c_s}}\int (\Delta|u_{\lambda}|^2-2|\nabla u_{\lambda}|^2)\Delta\varphi dx\\
=&-\frac{1}{\sqrt{c_s}}\Re\int |\Delta u_{\lambda}|^2\Delta\varphi dx
+\frac{\lambda}{2\sqrt{c_s}}\int |u_{\lambda}|^2\Delta^2\varphi dx-\frac{\lambda}{\sqrt{c_s}}\int |\nabla u_{\lambda}|^2\Delta\varphi dx,
\end{split}
\end{equation}
and
\begin{equation}\label{2}
\begin{split}
&\Re\int \Delta u_{\lambda}\nabla\ov{u}\cdot\nabla\varphi dx=
\frac{1}{\sqrt{c_s}}\int \Delta u_{\lambda}\nabla((-\Delta+\lambda)\ov{u_{\lambda}})\cdot{\nabla \varphi}dx\\
=&-\frac{\lambda}{\sqrt{c_s}}\int \nabla u_{\lambda}\cdot\nabla(\nabla\ov{u_{\lambda}}\cdot\nabla\varphi)dx+\frac{1}{2\sqrt{c_s}}\int |\Delta u_{\lambda}|^2\Delta\varphi dx.
\end{split}
\end{equation}
Similarly,
\begin{equation}\label{3}
\begin{split}
&\Re\int \nabla\ov{u_{\lambda}}\cdot\nabla (\nabla u\cdot\nabla\varphi)dx\\
=&\frac{1}{\sqrt{c_s}}\Re\int
\nabla\ov{u_{\lambda}}\cdot\nabla(\nabla(-\Delta+\lambda)u_{\lambda}\cdot\nabla \varphi)dx\\
=&-\frac{1}{2\sqrt{c_s}}\Re\int |\Delta u_{\lambda}|^2\Delta\varphi dx-\frac{\lambda}{\sqrt{c_s}}\Re\int \Delta \ov{u_{\lambda}}\nabla u_{\lambda}\cdot\nabla\varphi dx.
\end{split}
\end{equation}
Combining \eqref{1}, \eqref{2} and \eqref{3} gives
\begin{equation}\label{4}
\begin{split}
&2\Im\int (i(-\Delta)^s\ov{u})\nabla u\cdot\nabla \varphi dx+
2\Im \int \ov{u}\nabla (-i(-\Delta)^s u)\cdot\nabla \varphi\\
=&4\Re\int_0^{\infty}\lambda^{s}d\lambda\int \nabla \ov{u_{\lambda}}\cdot\nabla(\nabla u_{\lambda}\cdot\nabla\varphi)dx+
2\int_0^{\infty}\lambda^{s}d\lambda\int |\nabla u_{\lambda}|^2\Delta\varphi dx\\
&-\int_0^{\infty}\lambda^{s}d\lambda\int |u_{\lambda}|^2\Delta^2\varphi dx.
\end{split}
\end{equation}
The rest computation are exactly the same as for NLS. Indeed, put $v=u_{\lambda}$, and then
\begin{equation}\label{5}
\begin{split}
\Re\int \nabla \ov{v}\cdot(\nabla v\cdot\nabla \varphi)dx&=
\Re\int \partial_k\ov{v}(\partial^2_{jk}\varphi\partial_jv+\partial^2_{jk}v\partial_j\varphi )dx\\
&=\int \partial^2_{jk}\varphi \partial_j v\partial_k\ov{v}dx-\Re\int \partial_j(\partial_k\ov{v}\varphi_j) \partial_k vdx\\
&=\int \partial^2_{jk}\varphi \partial_j v\partial_k\ov{v}dx
-\Re\int\partial_kv \partial^{jk}\ov{v}\partial_j \varphi dx\\
&~~~-\int |\nabla v|^2\Delta\varphi dx\\
&=\int \partial^2_{jk}\varphi \partial_j v\partial_k\ov{v}dx
-\frac{1}{2}\int |\nabla v|^2\Delta\varphi dx.
\end{split}
\end{equation}
Plugging \eqref{5} into \eqref{4} yields \eqref{virial1}.

\subsection{Defocusing case}
We can also apply the virial identity to the defocusing FNLS:
\begin{align} \label{defocusing}
\begin{cases}    i\partial_tu-(-\Delta)^su= +|u|^{p-1} u,\quad
(t,x)\in\R\times\R^d,
\\
u(0,x)=u_0(x).
\end{cases}
\end{align}
With the help of the strategy in \cite{Lenzmann}, we can easily establish the scattering result for defocusing FNLS.
\begin{proposition}\label{defo}
Assume that $d\geq 3$, $s\in\left(\frac{d}{2d-1},1\right)$ and $s_c=\frac{d}{2}-\frac{2s}{p-1}\in (0,s)$.  Suppose $u_0\in H^s(\mathbb{R}^d)$ is radial. Then the solution to \eqref{defocusing} with initial data $u_0$ is global and scatters in $H^s(\mathbb{R}^d)$.
\end{proposition}
\begin{proof}
	For defocusing FNLS, the global well-posedness is ensured by energy conservation since $s_c<s$.
	
	It suffices to prove the scattering. Thanks to Lemma \ref{scacri}, we need exclude the concentration of mass.
	
	The virial identity in the defocusing case becomes
		\begin{equation}\label{vir}
	\begin{split}
	\frac{dM_{\varphi}}{dt}=&\int_{0}^{\infty}\lambda^{s}d\lambda\int_{\mathbb{R}^3} \left(4\partial_j\partial_k \varphi\partial_j\ov{u_{\lambda}}\partial_ku_{\lambda}-\Delta^2\varphi |u_{\lambda}|^2\right)dx\\
	&+\int_{\mathbb{R}^3}\Delta\varphi |u|^{p+1}dx.
	\end{split}
	\end{equation}
	
	In this case, we can simply take the test function $\varphi=\varphi_{\epsilon}(x)=\sqrt{|x|^2+\epsilon^2}$ and let $\epsilon\rightarrow 0$ followed by applying \eqref{vir} for $\varphi_{\epsilon}$ to obtain
	\begin{equation*}
	\begin{split}
	\int_{\mathbb{R}}\int_{\mathbb{R}^3}\frac{|u(t,x)|^{p+1}}{|x|}dxdt\lesssim 1,
	\end{split}
	\end{equation*}
	since $\|M_{\varphi_{\epsilon}}\|_{L_t^{\infty}}\lesssim 1$, uniformly in $\epsilon.$

From radial Sobolev embedding Lemma \ref{Strauss}
$$ |x|^{\frac{d-2s}{2}}|u(x)|\lesssim \|u\|_{\dot{H}_x^s},
$$
we have
\begin{equation*}
\begin{split}
\int_{\mathbb{R}}\int_{\mathbb{R}^3}|u(t,x)|^{p+1+\frac{2}{d-2s}}dtdx\leq
\int_{\mathbb{R}}\int_{\mathbb{R}^3}\frac{|u(t,x)|^{p+1}}{|x|}|x||u(t,x)|^{\frac{2}{d-2s}}dtdx\lesssim 1.
\end{split}
\end{equation*}
Letting $\frac{1}{q}=\frac{1}{r}=\left(\frac{d-2s}{(d-2s)(p+1)+2}\right)$, one has
$$ \|u\|_{L_{t,x}^{q}}\lesssim 1.
$$
The scattering will follow from the standard argument of interpolation and Strichartz inequality once we have showed that $(q,q)\in \Lambda_d(\gamma)$ for some $\gamma\in(0,s)$.

Indeed,  $(q,q)\in\Lambda_d(\gamma)$ for some $0<\gamma<s$ is equivalent to
$$ \frac{d-2s}{2(d+2s)}<\frac{1}{q}<\frac{d}{d+2s}.
$$
Now $s_c<s$ implies that $p<1+\frac{4s}{d-2s}$ and one can easily verify that
$$ \frac{1}{q}=\frac{d-2s}{(d-2s)(p+1)+2}\geq \frac{d-2s}{2(d+2s)}.
$$
The other side $s_c>0$ implies that
$$ p>1+\frac{4s}{d}\geq 1+\frac{4s}{d}-\frac{2}{d-2s}.
$$
From this, we can deduce that
$$ \frac{1}{q}=\frac{d-2s}{(d-2s)(p+1)+2}\leq \frac{d}{2(d+2s)}.
$$
\end{proof}
\subsection*{Acknowledgement} The authors C. Sun and J. Zheng are financed by ERC project SCAPDE. The author H. Wang is financed by China National Science
Foundation under the grant number 11101172, 11371158 and 11571131. The author X. Yao is financed by China National Science Foundation under the grant number 11371158.

\begin{center}

\end{center}

\end{document}